\documentclass[a4paper, 12pt, reqno,final]{amsart}
\usepackage[foot]{amsaddr} 
\usepackage[utf8x]{inputenc}
\usepackage[T1]{fontenc}
\usepackage[english]{babel}
\usepackage{amsmath, amsfonts, amsthm, amssymb, amscd}
\usepackage[final]{graphicx}
\usepackage[below]{placeins}
\usepackage{subfig}
\usepackage{multirow}
\usepackage{mathtools}

\usepackage{showkeys} 
\usepackage{esint}
\usepackage{caption}
\usepackage{dsfont}

\usepackage{tikz}
\usetikzlibrary{patterns}
\usetikzlibrary{shapes.geometric}

\usepackage[]{todonotes}
\usepackage{diagbox}

\newtheorem{theorem}{Theorem}[section]
\newtheorem{lemma}[theorem]{Lemma}
\newtheorem*{lemma*}{Lemma}

\numberwithin{equation}{section}



\makeatletter
\newcommand{\labitem}[2]{%
\def\@itemlabel{\textbf{#1}}
\item
\def\@currentlabel{#1}\label{#2}}
\makeatother

\newcommand{\norm}[1]{\left\|{#1}\right\|}
\newcommand{\abs}[1]{\left|{#1}\right|}

\newcommand{\rkla}[1]{{\left(#1\right)}}
\newcommand{\trkla}[1]{{(#1)}}
\newcommand{\gkla}[1]{{\left\{#1\right\}}}
\newcommand{\tgkla}[1]{{\{#1\}}}

\newcommand{\ekla}[1]{{\left[#1\right]}}
\newcommand{\tekla}[1]{{[#1]}}
\newcommand{\tabs}[1]{|{#1}|}

\newcommand{\bs}[1]{\boldsymbol{#1}}

\newcommand{\Om}{\mathcal{O}}
\newcommand{\iO}{\int_{\Om}}

\newcommand{\dx}{\, \mathrm{d}{x}}

\newcommand{\dt}{\, \mathrm{d}t}
\newcommand{\ds}{\, \mathrm{d}s}

\newcommand{\dW}{\, \mathrm{d}W}
\newcommand{\dbeta}{\, \mathrm{d}\beta}

\newcommand{\ids}{I}

\newcommand{\nn}{^{n}}

\newcommand{\no}{^{n-1}}

\newcommand{\h}{_{h}}

\newcommand{\weakstartop}{{\mathrm{weak}-(*)}}

\newcommand{\Uh}{U_{h}}

\newcommand{\Th}{\mathcal{T}_h}

\newcommand{\Ihop}{\mathcal{I}_h}
\newcommand{\Ih}[1]{\Ihop\gkla{#1}}

\newcommand{\restr}[2]{\ensuremath{
  \left.\kern-\nulldelimiterspace 
  #1 
  \vphantom{\big|} 
  \right|_{#2} 
  }}
  
\newcommand{\extend}[2]{\ensuremath{
  \left.\kern-\nulldelimiterspace 
  #1 
  \vphantom{\big|} 
  \right|^{#2} 
  }}
  
\newcommand{\diam}{\operatorname{diam}}

\newcommand{\g}[1]{\mathfrak{g}_{#1}}

\newcommand{\p}{{\mathfrak{p}}}

\newcommand{\expected}[1]{\mathds{E}\ekla{#1}}

\newcommand{\Prob}{\mathds{P}}

\newcommand{\projop}{\mathcal{P}_{\Uh}}
\newcommand{\proj}[1]{\projop\gkla{#1}}

\newcommand{\per}{{\operatorname{per}}}

\DeclareMathOperator*{\esssup}{ess\,sup}

\newcommand{\Ito}{It\^{o}}

\newcommand{\sinc}[1]{\blacktriangle^{#1}\boldsymbol{\xi}^\tau}

\newcommand{\err}{\mathfrak{e}\h}

\newcommand{\phicont}{u}

\newcommand{\mytitle}{Strong error estimates for a fully discrete SAV scheme for the stochastic Allen--Cahn equation with multiplicative noise}
\newcommand{\myshorttitle}{Error estimates for a fully discrete SAV scheme}

\usepackage[hidelinks]{hyperref}
\hypersetup{
  pdfauthor = {Stefan Metzger},
  pdftitle = {\myshorttitle}
}

\begin{document}
\title[\myshorttitle]{\mytitle}
\date{\today}
\author[S.~Metzger]{Stefan Metzger}
\address[S.~Metzger]{Friedrich--Alexander Universität Erlangen--Nürnberg,~Cauerstraße 11,~91058~Erlangen,~Germany}
\email{stefan.metzger@fau.de}


\keywords{stochastic Allen--Cahn equation, multiplicative noise, finite elements, strong rate of convergence, scalar auxiliary variable}
\subjclass[2010]{60H35, 65M60, 60H15, 65M12}



%
\selectlanguage{english}

\allowdisplaybreaks
\begin{abstract}
We analyze a fully discrete numerical scheme for the approximation of solutions to the stochastic Allen--Cahn equation with multiplicative noise.
The scheme uses a recently proposed augmented version of the scalar auxiliary variable method (cf.~[S.~Metzger, 2023, A convergent stochastic scalar auxiliary variable method, IMA J.~Numer.~Anal.]) 
\end{abstract}
\maketitle

\section{Introduction}\label{sec:introduction}
We analyze a recently suggested fully discrete, linear finite element scheme for the approximation of solutions to the stochastic Allen--Cahn equation with multiplicative noise (cf.~\cite{Metzger2024}).
On a periodic domain $\Om:=\trkla{0,1}^d\subset\mathds{R}^d$ ($d\in\tgkla{1,2,3}$), a time interval $\tekla{0,T}$, and a filtered probability space $\trkla{\Omega,\mathcal{A},\mathcal{F},\Prob}$ with a filtration $\mathcal{F}=\trkla{\mathcal{F}_t}_{t\in\tekla{0,T}}$, this equation reads as follows:\\
Find a stochastic process $\phicont\,:\,\Omega\times\tekla{0,T}\times\Om\rightarrow\mathds{R}$ satisfying
\begin{align}\label{eq:abstract:allencahn}
\mathrm{d}\phicont +\trkla{-\varepsilon\Delta\phicont +\varepsilon^{-1} F^\prime\trkla{\phicont}}\dt=\Phi\trkla{\phicont}\dW
\end{align}
with space-periodic boundary conditions.
Here, $\Phi$ is a $\phicont$-dependent Hilbert-Schmidt operator and $W$ denotes a cylindrical Wiener process.
The precise properties of $\Phi$ and $W$ will be specified in Section \ref{sec:notation}.\\
The deterministic version of \eqref{eq:abstract:allencahn}, the Allen--Cahn equation, describes the dynamics of a diffuse interface whose width is governed by the parameter $\varepsilon>0$.
This equation can be derived as the $L^2$-gradient flow of the Helmholtz free energy functional
\begin{align}\label{eq:def:energy}
\mathcal{E}\trkla{\phicont}:=\tfrac\varepsilon2\iO\abs{\nabla\phicont}^2\dx+\varepsilon^{-1}\iO F\trkla{\phicont}\dx\,
\end{align}
and satisfies the energy equality
\begin{align}\label{eq:energy:equality}
\mathcal{E}\trkla{\phicont}\big\vert_{t=T}+\int_0^T\iO\abs{-\varepsilon\Delta\phicont+\varepsilon^{-1}F^\prime\trkla{\phicont}}^2\dx\dt=\mathcal{E}\trkla{\phicont}\big\vert_{t=0}\,.
\end{align}
Here, $F$ denotes a double-well potential with minima in (or close to) the pure phases indicated by $\phi=\pm1$.
In this work, we shall discuss the case of positive polynomial double-well potentials, i.e.~we will consider $F\trkla{\phi}=\tfrac14\trkla{\phi^2-1}^2+\gamma$ with $\gamma>0$.
Here, shifting the double-well potential by adding the constant $\gamma$ does not change the dynamics of the resulting partial differential equation.
Yet, for the formulation of our numerical scheme, we will require $F>0$, i.e.~$\gamma$ can be seen as a numerical parameter.\\
As the equality \eqref{eq:energy:equality} constitutes the basis for stability and regularity results, satisfying a discrete analogon of \eqref{eq:energy:equality} is an important property of numerical schemes. 
This property, however, typically requires an implicit discretization of the nonlinear terms (see e.g.~the approaches proposed in \cite{Elliott1993} or \cite{Eyre98} which use convex-concave splittings or difference quotients for the approximation of $F^\prime$).
To reduce the computational costs linearization techniques like the invariant energy quadratization (IEQ) approach (see e.g.~\cite{Yang2016}, \cite{YangYu2018}, \cite{YangZhang20}, and the references therein) or the scalar auxiliary variable (SAV) method (cf.~\cite{ShenXuYang2018}, \cite{ShenXu18}, \cite{BouchritiPierreAlaa20}) were proposed in the recent years.
Both of these approaches introduce new auxiliary variables $q:=\sqrt{F\trkla{\phicont}}$ (IEQ) or $r:=\sqrt{\iO F\trkla{\phicont}\dx}$ (SAV).
With these additional auxiliary variables it is possible to derive linear schemes that are still stable with respect to the modified energies
\begin{subequations}
\begin{align}
&&\mathcal{E}_{\mathrm{IEQ}}\trkla{\phicont,q}&:=\tfrac\varepsilon2\iO\abs{\nabla\phicont}^2\dx +\varepsilon^{-1}\iO \abs{q}^2\dx\\
\text{~or~}&&\mathcal{E}_{\mathrm{SAV}}\trkla{\phicont,r}&:=\tfrac\varepsilon2\iO\abs{\nabla\phicont}^2\dx +\varepsilon^{-1}\abs{r}^2\,.
\end{align}
\end{subequations}
Although, these energies are only approximations of the original energy \eqref{eq:def:energy}, stability with respect to the modified energy is still sufficient to obtain analytical convergence results (see e.g.~\cite{ShenXu18}, \cite{Metzger2023}, \cite{LamWang2023}).
The development of improved versions of these schemes is still ongoing:
Variants of the SAV approach that aim for a better approximation of the energy can be found e.g.~in \cite{HouAzaiezXu19}, \cite{LiuLi20}, and \cite{ChengLiuShen2020}.
In this work, we analyze a recently proposed augmented version of the SAV method that is applicable to stochastic PDEs (cf.~\cite{Metzger2024}).\\

As the numerical investigation of stochastic partial differential equations typically requires the simulation of multiple paths, i.e.~the simulations have to be repeated multiple times using different independent sequences of (pseudo-)random numbers, the computational complexity of the employed numerical schemes becomes an even more prominent factor in the stochastic setting.
While for Lipschitz continuous nonlinearities even completely explicit schemes are viable (see e.g.~\cite{Hausenblas2002}, \cite{Hausenblas2003}, and \cite{GyongyMillet2009}), divergence of purely explicit discretizations was shown in \cite{HutzenthalerJentzenKloeden2011} (see also the preprint \cite{beccariEtAl2019_arxiv}) for stochastic differential equations with non-globally Lipschitz coefficients.
Consequently, the focus shifted towards implicit or semi-implicit discretizations (see e.g.~\cite{DeBouardDebussche2004}, \cite{JentzenKloedenWinkel2011}, \cite{BlomkerJentzen2013}, \cite{AntonopoulouBanasNurnbergProhl2021}) and towards tamed or truncated schemes (cf.~\cite{HutzenthalerJentzen2015}, \cite{Gyongy2016}, \cite{JentzenPusnik2019}, \cite{HutzenthalerJentzen2020}, \cite{QiAzaiezHuangXu2022}, and \cite{BeckerGessJentzenKloeden2023}).
Here, the latter approaches still treat the nonlinearities explicitly, but introduce a time step size dependent regularization which limits the influence of the nonlinearity.\\
The application of these discretization approaches to the stochastic Allen--Cahn equation was extensively studied.
Error estimates for tamed and truncated schemes for the stochastic Allen--Cahn equation with additive white noise can be found e.g.~in \cite{BeckerJentzen2019}, \cite{Wang2020}, \cite{CaiGanWang2021}, and \cite{BeckerGessJentzenKloeden2023}, while splitting or semi-implicit schemes for this setting were studied e.g.~in \cite{BrehierCuiHong2018}, \cite{BrehierGoudenege2019}, and \cite{BrehierGoudenege2020}.
Error estimates for fully implicit schemes with additive noise can be found for instance in \cite{KovacsLarssonLindgren2015}, \cite{KovacsLarssonLindgren2018}, \cite{LiuQiao2019}, \cite{QiWang2019}, and the references therein.
The case of multiplicative noise was studied e.g.~in \cite{FengLiZhang2017}, \cite{MajeeProhl2018}, \cite{LiuQiao2021}, \cite{HuangShen2023}, \cite{KruseWeiske2023}, and \cite{BreitProhl2024}.
In \cite{MajeeProhl2018}, the convergence properties of an implicit discretization of the stochastic Allen--Cahn equation were analyzed.
For discrete versions of the $L^\infty\trkla{0,T;L^2\trkla{\Omega;L^2\trkla{\Om}}}$-norm and the $L^2\trkla{\Omega;L^2\trkla{0,T;H^1\trkla{\Om}}}$-norm they established the optimal convergence order of $1/2-\delta$ (with arbitrary $\delta>0$) with respect to time and $1$ with respect to space.
In \cite{LiuQiao2021} convergence order $1/2$ with respect to time was shown for a drift-implicit Euler scheme and order $1$ was established for a Milstein scheme.
Tamed schemes for multiplicative noise were investigated in \cite{HuangShen2023} where convergence order $1/2$ in time was shown for the stochastic Allen--Cahn equation in one spatial dimension.
The same order of convergence was shown in \cite{KruseWeiske2023} for a BDF2-Maruyama method in three spatial dimensions.
Gradient-type multiplicative noise was investigated in \cite{FengLiZhang2017} where also convergence order $1/2$ with respect to time was established.\\

As SAV methods also provide a pathway to linear and unconditionally energy-stable schemes, they seem to be a natural candidate for the efficient treatment of stochastic PDEs.
However, although the SAV method was successfully applied in \cite{CuiHongSun2022_arxiv} to the stochastic wave equation, for most other stochastic PDEs its straightforward application is not successful.
It was for instance shown in \cite{Metzger2024} that numerical simulations of the stochastic Allen--Cahn equation performed using the SAV scheme without augmentation terms provide erroneous solutions that do not converge towards the correct limit.
These convergence problems of standard SAV schemes can be traced back to the poor temporal regularity of solutions to stochastic PDEs.
In particular, after discretizing in time the auxiliary variable $r\trkla{t\nn}$ does not coincide with $\sqrt{\iO F\trkla{\phicont\trkla{t\nn}}\dx}$ anymore.
To quantify the error, one can use computations similar to the ones presented in the proof of Lemma \ref{lem:regdisc} (see also \cite[Lemma 6.3]{Metzger2024} and \cite[Lemma 6.4]{Metzger2024_arxivb}) and show that the standard SAV update rule can be obtained as the first order terms of the Taylor expansion of $\sqrt{\iO F\trkla{\phicont\trkla{t\nn}}\dx}$ around $\phicont\trkla{t\no}$.
Hence, the error of this approximation in each time step is given as the remainder of this Taylor expansion, which roughly scales with $\tabs{\phicont\trkla{t\nn}-\phicont\trkla{t\no}}^2$.
Consequently, convergence results can be achieved if $\phicont$ is Hölder-continuous with a Hölder exponent larger than $1/2$.
As Brownian motions are only Hölder continuous with an exponent smaller than $1/2$, solutions to \eqref{eq:abstract:allencahn} will not possess the regularity required to establish convergence of the auxiliary variable.
To overcome this issue the approximation quality of the scalar auxiliary variable needs to be improved.
In \cite{Metzger2024} it was suggested by the author to add suitable linearizations of the second order terms of the Taylor expansion to the update rule leading to error terms scaling roughly with $\abs{\phicont\trkla{t\nn}-\phicont\trkla{t\no}}^{3}$ (see e.g.~Lemma \ref{lem:regdisc}).
This augmentation reduces the regularity requirements and hence allows to establish convergence even for stochastic PDEs with less regular solutions.
In particular, this discretization technique was successfully applied to stochastic Cahn--Hilliard equations with dynamic boundary conditions describing contact line tension (cf.~\cite{Metzger2024_arxivb}, where convergence along sequences was established).
Numerical simulations of the stochastic Allen--Cahn equation presented in \cite{Metzger2024} underline the practicality of the augmented SAV scheme:
While the numerical results coincide with the ones obtained using the time discretization analyzed in \cite{MajeeProhl2018}, their computation required only half of the time.
In this publication, we derive strong rates of convergence for this augmented SAV scheme for the stochastic Allen--Cahn equation and show that this linear, unconditionally stable discrete scheme enjoys the same optimal convergence rates as the implicit scheme proposed in \cite{MajeeProhl2018}.\\

The outline of the paper is as follows:
In Section \ref{sec:notation} we introduce the notation and state our assumptions on the given data.
The fully discrete scheme and the main result, which will be proven in the subsequent sections, can be found in Section \ref{sec:main}.
In Section \ref{sec:regularity}, we collect the regularity results available for solutions to the stochastic Allen--Cahn equation \eqref{eq:abstract:allencahn} and its discrete approximations.
Using these regularity results, we then quantify the approximation error in Section \ref{sec:error}.
We conclude by presenting numerical simulations underlining the practicality of our scheme in Section \ref{sec:numerics}.

\section{Notation and technical preliminaries}\label{sec:notation}
We consider the stochastic Allen--Cahn equation with periodic boundary conditions on the $d$-dimensional torus $\Om\equiv\trkla{0,1}^d\subset\mathds{R}^d$ with $d\in\tgkla{1,2,3}$.
We denote the space of $k$-times weakly differentiable periodic functions with weak derivatives in $L^p\trkla{\Om}$ by $W^{k,p}_{\per}\trkla{\Om}$.
For $p=2$, we denote the Hilbert spaces $W^{k,2}_{\per}\trkla{\Om}$ by $H^k_{\per}\trkla{\Om}$.
Given a Banach space $X$ and a set $I$, the symbol $L^p\trkla{I;X}$ ($p\in\tekla{1,\infty}$) stands for the Bochner space of strongly measurable $L^p$-integrable functions on $I$ with values in $X$.
If $X$ is only separable (and not reflexive), we shall follow the notation used in \cite[Chapter 0.3]{FeireislNovotny17} and denote the dual space of $L^{p/\trkla{p-1}}\trkla{I;X}$ by $L^p_{\weakstartop}\trkla{I;X^\prime}$.
The symbol $C^{k,\alpha}\trkla{I;X}$ with $k\in\mathds{N}_0$ and $\alpha\in(0,1]$ denotes the space of $k$-times continuously differentiable functions from $I$ to $X$ whose $k$th derivatives are Hölder continuous with Hölder exponent $\alpha$.
If $I=X=\mathds{R}$, we shall write $C^{k,\alpha}\trkla{\mathds{R}}$.\\

Concerning the discretization with respect to time, we shall assume that
\begin{itemize}
\labitem{\textbf{(T)}}{item:T} the time interval $I:=\tekla{0,T}$ is subdivided into $N$ equidistant intervals $I_n$ given by nodes $\trkla{t\nn}_{n=0,\ldots,N}$ with $t^0=0$, $t^N=T$, and $t^{n+1}-t\nn=\tau=\tfrac{T}{N}$.
\end{itemize}
The spatial discretization is based on partitions $\Th$ of $\Om$ depending on a discretization parameter $h>0$ satisfying the following assumption:
\begin{itemize}
\labitem{\textbf{(S)}}{item:S} The family $\gkla{\Th}_{h>0}$ is a quasi-uniform family of partitions of $\Om$ into disjoint, open simplices $K$ which satisfy
\begin{align*}
\overline{\Om}\equiv \bigcup_{K\in\Th} \overline{K}&&\text{with~} \max_{K\in\Th}\diam\trkla{K}\leq h\,.
\end{align*}
\end{itemize}
In our discrete scheme we will approximate the solutions using continuous, piecewise linear finite element functions on $\Th$.
The corresponding function space will be denoted by $\Uh$ and is spanned by functions $\tgkla{\chi_{h,k}}_{k=1,\ldots,\dim\Uh}$ forming a dual basis to the vertices $\tgkla{\bs{x}_{h,k}}_{k=1,\ldots,\dim\Uh}$.
The nodal interpolation operators $\Ihop$ mapping $C^0_{\per}\trkla{\overline{\Om}}$ onto $\Uh$ is defined by
\begin{align}
\Ih{a}:=\sum_{k=1}^{\dim\Uh}a\trkla{\bs{x}_{h,k}}\chi_{h,k}\,.
\end{align}
This interpolation operator allows to define equivalent versions of the standard $L^p\trkla{\Om}$-norms.
In particular, we have
\begin{align}\label{eq:normequivalence}
c\norm{\zeta\h}_{L^p\trkla{\Om}}\leq \rkla{\iO\Ih{\abs{\zeta\h}^p}\dx}^{1/p}\leq C\norm{\zeta\h}_{L^p\trkla{\Om}}
\end{align}
for $p\in[1,\infty)$ and $\zeta\h\in\Uh$ with constants $c,C>0$ independent of $h$.
This estimate follows from a combination of Jensen's inequality and an inverse inequality (see e.g.~\cite[Lemma 3.2.11]{SieberDiss2021}).
For future reference, we recall the following error estimates for $\Ihop$:
\begin{lemma}\label{lem:Ihestimates}
Let $\Th$ satisfy assumption \ref{item:S} and let $p\in[1,\infty)$ and $1\leq q\leq\infty$.
Then, for $q^*=\tfrac{q}{q-1}$ (with the obvious modifications for $q=1$ and $q=\infty$) the estimates
\begin{subequations}
\begin{align}
\norm{\trkla{\ids-\Ihop}\gkla{f\h g\h}}_{L^p\trkla{\Om}}&\leq Ch^2\norm{\nabla f\h}_{L^{pq}\trkla{\Om}}\norm{\nabla g\h}_{L^{pq^*}\trkla{\Om}}\,,\\
\norm{\trkla{\ids-\Ihop}\gkla{f\h g\h}}_{W^{1,p}\trkla{\Om}}&\leq Ch\norm{\nabla f\h}_{L^{pq}\trkla{\Om}}\norm{\nabla g\h}_{L^{pq^*}\trkla{\Om}}
\end{align}
\end{subequations}
hold true for all $f\h,g\h\in\Uh$.
\end{lemma}
For the proof we refer e.g.~to \cite[Lemma 2.1]{Metzger2020}.
The $L^2$-projection from $L^2\trkla{\Om}$ to $\Uh$ will be denoted by $\projop$.
This projection is stable with respect to the full $H^1$-norm and the $H^1$-seminorm and satisfies the error estimates
\begin{subequations}\label{eq:error:proj}
\begin{align}
\norm{f-\proj{f}}_{L^2\trkla{\Om}}&\leq Ch\norm{f}_{H^1\trkla{\Om}}\,,\\
\norm{g-\proj{g}}_{L^2\trkla{\Om}}+h\norm{\nabla\trkla{g-\proj{g}}}_{L^2\trkla{\Om}}&\leq Ch^2\norm{\Delta g}_{L^2\trkla{\Om}}
\end{align}
\end{subequations}
for all $f\in H^1_\per\trkla{\Om}$ and $g\in H^2_{\per}\trkla{\Om}$ (cf.~\cite{Ern2004}).
To approximate the Laplacian we define the discrete Laplacian $\Delta\h\,:\Uh\rightarrow \Uh$ via
\begin{align}\label{eq:defdisclap}
\iO\Ih{\Delta\h\zeta\h\psi\h}\dx=-\iO\nabla \zeta\h\cdot\nabla\psi\h\dx&&\text{for~}\zeta\h,\psi\h\in\Uh\,.
\end{align}
For simplicity, we assume that the continuous and discrete initial data are deterministic and satisfy
\begin{itemize}
\labitem{\textbf{(I)}}{item:I} $\phicont_0\in H^2_{\per}\trkla{\Om}$ and $\phi\h^0:=\Ih{\phicont_0}\in\Uh$. 
\end{itemize}
We shall now collect our assumptions on the stochastic source term on the right-hand side of \eqref{eq:abstract:allencahn}, which is governed by a $\mathcal{Q}$-Wiener process defined on a filtered probability space $\trkla{\Omega,\mathcal{A},\mathcal{F},\Prob}$ with a complete right-continuous filtration $\mathcal{F}$ and an operator $\Phi$ mapping into the space of Hilbert-Schmidt operators.
We shall assume that
\begin{itemize}
\labitem{\textbf{(W1)}}{item:W1} $\mathcal{Q}\,:\,L^2\trkla{\Om}\rightarrow L^2\trkla{\Om}$ is a trace class operator with the representation
\begin{align}
\mathcal{Q}g:=\sum_{k\in\mathds{Z}} \lambda_k^2\trkla{g,\g{k}}_{L^2\trkla{\Om}}\g{k}&&\text{for~} g\in L^2\trkla{\Om}\,,
\end{align}
where $\trkla{.,.}_{L^2\trkla{\Om}}$ denotes the $L^2$-inner product, $\trkla{\lambda_k}_{k\in\mathds{Z}}$ are given positive real numbers and $\trkla{\g{k}}_{k\in\mathds{Z}}$ is an orthonormal basis of $L^2\trkla{\Om}$ consisting of smooth periodic functions.
\end{itemize}
Hence, $\mathcal{Q}^{1/2}$, which is given by $\mathcal{Q}^{1/2} g:= \sum_{k\in\mathds{Z}}\lambda_k\trkla{g,\g{k}}_{L^2\trkla{\Om}}\g{k}$, is a Hilbert--Schmidt operator from $L^2\trkla{\Om}$ to $L^2\trkla{\Om}$.
We shall denote its image of $L^2\trkla{\Om}$ by
\begin{align}
\mathcal{Q}^{1/2}L^2\trkla{\Om}:=\gkla{\mathcal{Q}^{1/2}g\,:\,g\in L^2\trkla{\Om}}\,.
\end{align}
Thus, using a family $\trkla{\beta_k}_{k\in\mathds{Z}}$ of mutually independent Brownian motions, we can write the $\mathcal{Q}$-Wiener process $\trkla{W\trkla{t}}_{t\in\tekla{0,T}}$ as
\begin{align}
W\trkla{t}:=\sum_{k\in\mathds{Z}}\lambda_k\g{k}\beta_k\trkla{t}\,.
\end{align}
Furthermore, we shall assume that the $\mathcal{Q}$-Wiener process is colored in the sense that
\begin{itemize}
\labitem{\textbf{(W2)}}{item:W2} there exists a positive constant $\widehat{C}$ such that 
\begin{align}
\sum_{k\in\mathds{Z}}\lambda_k^2\norm{\g{k}}_{W^{2,\infty}\trkla{\Om}}^2\leq \widehat{C}\,.
\end{align}
\end{itemize}
The operator $\Phi$ maps a stochastic process $\phicont$ into the space of Hilbert--Schmidt operators from $\mathcal{Q}^{1/2}L^2\trkla{\Om}$ to $L^2\trkla{\Om}$ and is defined via
\begin{align}
\Phi\trkla{\phicont}g:=\varrho\trkla{\phicont}\sum_{k\in\mathds{Z}}\trkla{g,\g{k}}_{L^2\trkla{\Om}}\g{k}
\end{align}
for all $g\in\mathcal{Q}^{1/2}L^2\trkla{\Om}$.
Here, the coefficient function $\varrho\,:\,\mathds{R}\rightarrow\mathds{R}$ satisfies assumptions similar to the ones used in \cite{MajeeProhl2018} and \cite{BreitProhl2024}.
In particular, we shall assume
\begin{itemize}
\labitem{\textbf{(C)}}{item:C} $\varrho\in W^{2,\infty}\trkla{\mathds{R}}\cap C^2\trkla{\mathds{R}}$.
\end{itemize}
This assumption also implies that $\varrho\in C^{0,1}\trkla{\mathds{R}}$ and consequently allows for the estimate $\norm{\nabla\Ih{\varrho\trkla{\zeta\h}}}_{L^p\trkla{\Om}}\leq \overline{C}\norm{\nabla\zeta\h}_{L^p\trkla{\Om}}$ for $\zeta\h\in \Uh$ and $p\in\tekla{1,\infty}$ with a constant $\overline{C}$ depending on the Lipschitz constant of $\varrho$ but not on $h$.
In our discrete scheme, we shall only consider a finite number of modes.
Hence, we will approximate $\Phi$ by a discrete version $\Phi\h$ which is defined via
\begin{align}\label{eq:def:Phih}
\Phi\h\trkla{\zeta\h} f:=\sum_{k\in\mathds{Z}\h}\Ih{\varrho\trkla{\zeta\h}\trkla{f,\g{k}}_{L^2\trkla{\Om}}\g{k}}\in\Uh
\end{align}
for all $f\in\mathcal{Q}^{1/2}L^2\trkla{\Om}$ and $\zeta\h\in\Uh$.
Here, $\mathds{Z}\h$ is an $h$-dependent finite subset of $\mathds{Z}$ satisfying 
\begin{itemize}
\labitem{\textbf{(Z)}}{item:Z} $\sum_{k\in\mathds{Z}\setminus\mathds{Z}\h}\norm{\lambda_k\g{k}}_{L^2\trkla{\Om}}^2\leq C h^2$.
\end{itemize}

\section{Discrete scheme and main result}\label{sec:main}
In this section, we introduce our numerical scheme and formulate our main convergence result which shall be proven in the subsequent sections.
To simplify the notation, we set $\varepsilon=1$ and start from the following version of the stochastic Allen--Cahn equation:\\
Find a stochastic process $\phicont$ satisfying
\begin{align}\label{eq:allencahn}
\phicont\trkla{t} + \int_0^t\rkla{-\Delta\phicont + F^\prime\trkla{\phicont}}\ds = \phicont_0 +\int_0^t \Phi\trkla{\phicont}\dW&&\text{in~}L^2_{\per}\trkla{\Om}
\end{align}
$\Prob$-almost surely for all $t\in\tekla{0,T}$.
The existence of stochastically strong, pathwise unique solutions to \eqref{eq:allencahn} is well-known (see e.g.~\cite[Theorem 2.1]{BreitProhl2024}). 
The relevant regularity results for these solutions will be collected in Section \ref{sec:regularity} below.

To approximate \eqref{eq:allencahn} numerically we shall use the following fully discrete finite element scheme which is based on an augmented version of the scalar auxiliary variable method proposed in \cite{Metzger2024}:\\
Given $\Uh\times\mathds{R}$-valued random variables $\trkla{\phi\h\no,r\h\no}$ find $\Uh\times\mathds{R}$-valued random variables $\trkla{\phi\h\nn,r\h\nn}$ which satisfy pathwise
\begin{subequations}\label{eq:discscheme}
\begin{align}\label{eq:discscheme:phi}
\begin{split}
\iO&\Ih{\trkla{\phi\h\nn-\phi\h\no}\psi\h}\dx +\tau\iO\nabla\phi\h\nn\cdot\nabla\psi\h\dx \\
&+\tau\frac{r\h\nn}{\sqrt{E\h\trkla{\phi\h\no}}}\iO\Ih{F^\prime\trkla{\phi\h\no}\psi\h}\dx\\
&-\left[\tau\frac{r\h\nn}{4\ekla{E\h\trkla{\phi\h\no}}^{3/2}}\iO\Ih{F^\prime\trkla{\phi\h\no}\Phi\h\trkla{\phi\h\no}\sinc{n}}\dx\iO\Ih{F^\prime\trkla{\phi\h\no}\psi\h}\dx\right.\\
&\qquad\left.-\tau\frac{r\h\nn}{2\sqrt{E\h\trkla{\phi\h\no}}}\iO\Ih{F^{\prime\prime}\trkla{\phi\h\no}\Phi\h\trkla{\phi\h\no}\sinc{n}\psi\h}\dx\right]\\
 =& \iO\Ih{\Phi\h\trkla{\phi\h\no}\sinc{n}\psi\h}\dx
\end{split}
\end{align}
for all $\psi\h\in\Uh$ and
\begin{align}\label{eq:discscheme:r}
\begin{split}
r\h\nn-r\h\no =&\, \frac{1}{2\sqrt{E\h\trkla{\phi\h\no}}}\iO\Ih{F^\prime\trkla{\phi\h\no}\trkla{\phi\h\nn-\phi\h\no}}\dx\\
&\,+\left[-\frac{1}{8\ekla{E\h\trkla{\phi\h\no}}^{3/2}}\iO\Ih{F^\prime\trkla{\phi\h\no}\Phi\h\trkla{\phi\h\no}\sinc{n}}\dx\right.\\
&\qquad\qquad\times\iO\Ih{F^\prime\trkla{\phi\h\no}\trkla{\phi\h\nn-\phi\h\no}}\dx\\
&\left.+\frac1{4\sqrt{E\h\trkla{\phi\h\no}}}\iO\Ih{F^{\prime\prime}\trkla{\phi\h\no}\trkla{\phi\h\nn-\phi\h\no}\Phi\h\trkla{\phi\h\no}\sinc{n}}\dx\right]\,.
\end{split}
\end{align}
\end{subequations}
Here, the abbreviations $\sinc{n}:=W\trkla{t\nn}-W\trkla{t\no}$ and $E\h\trkla{\zeta\h}:=\iO\Ih{F\trkla{\zeta\h}}\dx$ were used.
Due to the choice $F\trkla{\varphi}:=\tfrac14\trkla{\varphi^2-1}^2+\gamma$, we immediately obtain $E\h\trkla{\cdot}\geq \gamma>0$.
Following assumption \ref{item:I}, the discrete initial data $\phi\h^0$ is chosen as $\Ih{\phicont_0}$. 
For the scalar auxiliary variable, we use $r\h^0:=\sqrt{E\h\trkla{\phi\h^0}}$.\\
In comparison to the standard SAV method, the update formula \eqref{eq:discscheme:r} for $r\h\nn$ in the augmented SAV scheme is enhanced by the two additional terms in the square brackets. 
These additional terms can be interpreted as suitable linearizations of the second order terms of the Taylor expansion of $\sqrt{E\h\trkla{\phi\h\nn}}$ around $\phi\h\no$.
As shown in Lemma \ref{lem:regdisc} (see also \cite[Lemma 6.3]{Metzger2024}) these augmentation terms ensure the correct convergence behavior of $r\h\nn$, if the solution is not Hölder continuous with exponent larger than $1/2$.
To guarantee stability of the scheme, it is necessary to also add counterparts of these augmentation terms in \eqref{eq:discscheme:phi}.\\
To simplify the notation, we shall denote the additional terms in \eqref{eq:discscheme:phi}, which were added to the standard SAV method, by
\begin{align}
\begin{split}
\Xi\h\nn:=&\,-\frac{r\h\nn}{4\ekla{E\h\trkla{\phi\h\no}}^{3/2}}\iO\Ih{F^\prime\trkla{\phi\h\no}\Phi\h\trkla{\phi\h\no}\sinc{n}}\dx\Ih{F^\prime\trkla{\phi\h\no}}\\
&\,+\frac{r\h\nn}{2\sqrt{E\h\trkla{\phi\h\no}}}\Ih{F^{\prime\prime}\trkla{\phi\h\no}\Phi\h\trkla{\phi\h\no}\sinc{n}}\in\Uh\,.
\end{split}
\end{align}
Using this abbreviation, we can rewrite \eqref{eq:discscheme:phi} as
\begin{multline}\label{eq:discschemesimple}
\phi\h\nn +\int_{t\no}^{t\nn}\rkla{-\Delta\h\phi\h\nn +\frac{r\h\nn}{\sqrt{E\h\trkla{\phi\h\no}}}\Ih{F^\prime\trkla{\phi\h\no}} +\Xi\h\nn}\dt\\=\phi\h\no +\int_{t\no}^{t\nn}\Phi\h\trkla{\phi\h\no}\dW\,.
\end{multline}
We want to emphasize that due to \eqref{eq:def:Phih} all terms in \eqref{eq:discschemesimple} are finite element functions contained in $\Uh$.\\
We are now in the position to state an error estimate for solutions to the augmented SAV scheme \eqref{eq:discscheme} which is the main result of this work:
\begin{theorem}\label{thm:error}
Let the assumptions \ref{item:S}, \ref{item:T}, \ref{item:C}, \ref{item:I}, \ref{item:W1}, \ref{item:W2}, and \ref{item:Z} hold true.
Then, for every $\delta>0$, there exists a constant $0<C_\delta<\infty$ which is independent of the discretization parameters $\tau$ and $h$ such that for all $\tau$ sufficiently small the estimate
\begin{align*}
\max_{0\leq M \leq N}\expected{\norm{\phicont\trkla{t^M}-\phi\h^M}_{L^2\trkla{\Om}}^2}+\tau\sum_{n=1}^N\expected{\norm{\nabla\phicont\trkla{t\nn}-\nabla\phi\h\nn}_{L^2\trkla{\Om}}^2}\leq C_\delta\trkla{\tau^{1-\delta}+h^2}\,,
\end{align*}
holds true.
Here, $\phicont$ is the solution process to \eqref{eq:allencahn} and $\trkla{\phi\h\nn}_n$ solves \eqref{eq:discscheme}.
\end{theorem}
In the next section, we collect the necessary regularity results for $\phicont$ and $\trkla{\phi\h\nn}_n$. 
The proof of Theorem \ref{thm:error} can be found in Section \ref{sec:error}.
\section{Regularity of solutions to the continuous and discrete problem}\label{sec:regularity}
In this section, we collect the required regularity results for solutions to the continuous and discrete versions of the stochastic Allen--Cahn equation.
As shown e.g.~in \cite[Theorem 2.1 and Lemma 2.2]{BreitProhl2024}, \cite[Theorem 4.2]{Metzger2024}, and \cite[Lemma 3.2]{MajeeProhl2018}, solutions to \eqref{eq:allencahn} possess the following regularity:
\begin{lemma}\label{lem:regu}
Let $W$ be a Wiener process defined on $\trkla{\Omega,\mathcal{A},\mathcal{F},\Prob}$ satisfying \ref{item:W1} and \ref{item:W2}.
Furthermore, let the assumptions \ref{item:C} and \ref{item:I} hold true.
Then, there exists a unique, progressively $\mathcal{F}$-measurable process
\begin{align}
\begin{split}\label{eq:regu:basic}
\phicont\in L^{2\p}_{\weakstartop}&\,\trkla{\Omega;L^\infty\trkla{0,T;H^1_\per\trkla{\Om}}}\cap L^{2\p}\trkla{\Omega; C^{0,\trkla{\p-1}/\trkla{2\p}}\trkla{\tekla{0,T};L^2\trkla{\Om}}}\\
&\cap L^{2\p}\trkla{\Omega;L^2\trkla{0,T;H^2_{\per}\trkla{\Om}}}
\end{split}
\end{align}
for any $\p\in\trkla{1,\infty}$, which satisfies
\begin{align}\label{eq:existenceu}
\phicont\trkla{t}-\phicont_0+\int_0^t\trkla{-\Delta\phicont+F^\prime\trkla{\phicont}}\ds=\int_0^t\Phi\trkla{\phicont}\dW
\end{align}
$\Prob$-almost surely in $L^2\trkla{\Om}$ for all $t\in\tekla{0,T}$.\\
In addition, this process satisfies
\begin{subequations}
\begin{align}\label{eq:regu:improved}
\expected{\sup_{0\leq s\leq T}\norm{\nabla^2\phicont\trkla{s}}_{L^2\trkla{\Om}}^2}&\leq C\,,\\
\expected{\norm{\phicont\trkla{t}-\phicont\trkla{s}}_{H^1\trkla{\Om}}^2}&\leq C\abs{t-s}&\text{for all~}s,t\in\tekla{0,t}\,.\label{eq:regu:hoelder}
\end{align}
\end{subequations}
\end{lemma}

Existence and regularity of discrete solutions to \eqref{eq:discscheme} was studied in \cite{Metzger2024} under slightly less restricting assumptions and convergence along sequences was established.
Although \cite{Metzger2024} dealt with the case of homogeneous Neumann boundary conditions, the established results carry over to the periodic case.
For the readers convenience, we shall collect the relevant results:

\begin{lemma}\label{lem:regdisc}
Let the assumptions \ref{item:S}, \ref{item:T}, \ref{item:C}, \ref{item:I}, \ref{item:W1}, and \ref{item:W2} hold true.
Then, for given $\trkla{\tau,h}>0$, there exists a pathwise unique sequence $\trkla{\phi\h\nn,r\h\nn}_{n\geq1}$ which solves \eqref{eq:discscheme} for each $\omega\in\Omega$, such that the map $\trkla{\phi\h\nn,r\h\nn}\,:\,\Omega\rightarrow \Uh\times\mathds{R}$ is $\mathcal{F}_{t\nn}$-measurable.
These sequences satisfy the following uniform bounds:
\begin{subequations}\label{eq:discreteregularity}
\begin{align}\label{eq:regularity}
\begin{split}
&\expected{\max_{0\leq n\leq N}\norm{\phi\h\nn}_{H^1\trkla{\Om}}^{2\p}}+\expected{\max_{0\leq n\leq N}\abs{r\h\nn}^{2\p}}+\expected{\rkla{\sum_{n=1}^N\norm{\phi\h\nn-\phi\h\no}_{H^1\trkla{\Om}}^2}^\p}\\
&\qquad+\expected{\rkla{\sum_{n=1}^N\abs{r\h\nn-r\h\no}^2}^\p}+\expected{\rkla{\sum_{n=1}^N\tau\norm{\Delta\h\phi\h\nn}_{L^2\trkla{\Om}}^2}^\p}\\
&\qquad+\expected{\rkla{\sum_{n=1}^N\tau\norm{-\Delta\h\phi\h\nn+\frac{r\h\nn}{\sqrt{E\h\trkla{\phi\h\no}}}\Ih{F^\prime\trkla{\phi\h\no}}+\Xi\h\nn}_{L^2\trkla{\Om}}^2}^\p}\\
&\qquad +\tau^{-\p}\expected{\rkla{\sum_{n=1}^N\tau\norm{\Xi\h\nn}_{L^2\trkla{\Om}}^2}^\p}+\trkla{l\tau}^{-\p}\expected{\sum_{m=0}^{N-l}\tau\norm{\phi\h^{m+l}-\phi\h^m}_{L^2\trkla{\Om}}^{2\p}}\leq C\,,
\end{split}
\end{align}
\begin{align}
\expected{\max_{0\leq m\leq N}\abs{r\h^m-\sqrt{E\h\trkla{\phi\h^m}}}^\p}&\leq C \tau^{\p\trkla{1/2-\delta/2}}\,,\label{eq:errorr}
\end{align}
\end{subequations}
with $1\leq\p <\infty$, $l\in\tgkla{0,\ldots,N}$, and $0<\delta<1$ arbitrary.
\end{lemma}
\begin{proof}
The discrete existence result can be found in \cite[Lemma 4.1]{Metzger2024}.
The regularity results stated in \eqref{eq:regularity} were proven in \cite[Lemma 6.1, Corollary 6.1, Corollary 6.2, Lemma 6.2]{Metzger2024}.
The proof of \eqref{eq:errorr} follows the lines of \cite[Lemma 6.3]{Metzger2024}.
Yet, in \cite{Metzger2024} a more general class of potentials was considered leading to a smaller rate of convergence.
In particular, the result derived in \cite{Metzger2024} also covers the cases of $F\in C^{2,\nu}\trkla{\mathds{R}}$ ($\nu\in(0,1]$) and $F\in C^3\trkla{\mathds{R}}$ satisfying $\tabs{F^{\prime\prime\prime}\trkla{\zeta}}\leq C\trkla{1+\tabs{\zeta}^2}$.
As we here consider a polynomial double-well potential where $F^{\prime\prime\prime}$ has only linear growth, we shall briefly show how these differences result in the improved convergence rate.\\
As in the proof of \cite[Lemma 6.3]{Metzger2024}, we start with a Taylor expansion of $\sqrt{E\h\trkla{\phi\h\nn}}$, use that $\Phi\h\trkla{\phi\h\no}\sinc{n}=\trkla{\phi\h\nn-\phi\h\no}+\tau\mu\h\nn$ with 
\begin{align}
\mu\h\nn:=-\Delta\h\phi\h\nn +\frac{r\h\nn}{\sqrt{E\h\trkla{\phi\h\no}}}\Ih{F^\prime\trkla{\phi\h\no}} +\Xi\h\nn\,,
\end{align} 
and substitute \eqref{eq:discscheme:r}. 
Thus, we obtain
{\allowdisplaybreaks
\begin{align}\label{eq:taylorr}
\begin{split}
&\sqrt{E\h\trkla{\phi\h\nn}}-\sqrt{E\h\trkla{\phi\h\no}}=r\h\nn-r\h\no\\
&\quad-\frac{1}{4\sqrt{E\h\trkla{\phi\h\no}}}\iO\Ih{F^{\prime\prime}\trkla{\phi\h\no}\trkla{\phi\h\nn-\phi\h\no}\tau\mu\h\nn}\dx\\
&\quad+\frac{1}{4\sqrt{E\h\trkla{\phi\h\no}}}\iO\Ih{\trkla{F^{\prime\prime}\trkla{\varphi_1}-F^{\prime\prime}\trkla{\phi\h\no}}\trkla{\phi\h\nn-\phi\h\no}^2}\dx\\
&\quad+\frac{1}{8\tekla{E\h\trkla{\phi\h\no}}^{3/2}}\rkla{\iO\Ih{F^\prime\trkla{\phi\h\no}\trkla{\phi\h\nn-\phi\h\no}}\dx}\rkla{\iO\Ih{F^\prime\trkla{\phi\h\no}\tau\mu\h\nn}\dx}
\end{split}\\
\begin{split}\nonumber
&\quad-\frac{1}{8\tekla{E\h\trkla{\phi\h\no}}^{3/2}}\rkla{\iO\Ih{F^\prime\trkla{\phi\h\no}\trkla{\phi\h\nn-\phi\h\no}}\dx}\\
&\quad\qquad\qquad\times\rkla{\iO\Ih{F^{\prime\prime}\trkla{\varphi_2}\trkla{\phi\h\nn-\phi\h\no}^2}\dx}\\
&\quad-\frac{1}{8\tekla{E\h\trkla{\phi\h\no}}^{3/2}}\rkla{\tfrac12\iO\Ih{F^{\prime\prime}\trkla{\varphi_2}\trkla{\phi\h\nn-\phi\h\no}^2}\dx}^2\\
&\quad+\frac{1}{16\tekla{\widetilde{E\h}}^{5/2}}\rkla{\iO\Ih{F^\prime\trkla{\varphi_3}\trkla{\phi\h\nn-\phi\h\no}}\dx}^3\\
&=:r\h\nn-r\h\no +R_{1,n}+R_{2,n}+R_{3,n}+R_{4,n}+R_{5,n}+R_{6,n}\,
\end{split}
\end{align}}
with $\varphi_1,\varphi_2,\varphi_3\in\operatorname{conv}\tgkla{\phi\h\nn,\phi\h\no}$ and $\widetilde{E}\h\in\operatorname{conv}\tgkla{E\h\trkla{\phi\h\nn},E\h\trkla{\phi\h\no} }$.
Summing \eqref{eq:taylorr} from $n=1$ to $m\leq N$, noting that by definition $r\h^0=\sqrt{E\h\trkla{\phi\h^0}}$, and taking the $\p$-th power, the maximum over all $m$, and the expected value, we obtain
\begin{align}
\begin{split}
\expected{\max_{0\leq m\leq N}\abs{r\h^m -\sqrt{E\h\trkla{\phi\h^m}}}^\p}\leq&\,C\sum_{j=1}^6\expected{\rkla{\sum_{n=1}^N\abs{R_{j,n}}}^\p}\,.
\end{split}
\end{align}
It was already shown in \cite{Metzger2024} that $\expected{\rkla{\sum_{n=1}^N \abs{R_{j,n}}}^\p}\leq C\tau^{\p/2}$ for $j\in\tgkla{1,3,4,5,6}$.
Hence, it remains to take a closer look at the estimate for $\expected{\rkla{\sum_{n=1}^N \abs{R_{2,n}}}^\p}$:
By applying Hölder's inequality, the Gagliardo--Nirenberg inequality, and Young's inequality, we obtain for any $0<\delta<1$
\begin{align}
\begin{split}
&\abs{\iO\Ih{\trkla{F^{\prime\prime}\trkla{\varphi_1}-F^{\prime\prime}\trkla{\phi\h\no}}\trkla{\phi\h\nn-\phi\h\no}^2}\dx}\\
&\leq C\norm{\Ih{F^{\prime\prime\prime}\trkla{\widehat{\varphi}}}}_{L^6\trkla{\Om}}\norm{\phi\h\nn-\phi\h\no}_{L^{18/5}\trkla{\Om}}^3\\
&\leq C\rkla{\norm{\phi\h\nn}_{H^1\trkla{\Om}}+\norm{\phi\h\no}_{H^1\trkla{\Om}}}\norm{\phi\h\nn-\phi\h\no}_{H^1\trkla{\Om}}^{\delta+\trkla{2-\delta}}\norm{\phi\h\nn-\phi\h\no}_{L^2\rkla{\Om}}\\
&\leq C\rkla{\norm{\phi\h\nn}_{H^1\trkla{\Om}}^{1+\delta}+\norm{\phi\h\no}_{H^1\trkla{\Om}}^{1+\delta}}\left(\tau^{1/2-\delta/2} \norm{\phi\h\nn-\phi\h\no}_{H^1\trkla{\Om}}^2 \right.\\
&\qquad\qquad\qquad\left.+\tau^{-\trkla{1/\delta-3/2+\delta/2}}\norm{\phi\h\nn-\phi\h\no}_{L^2\trkla{\Om}}^{2/\delta}  \right)\,
\end{split}
\end{align}
with suitable $\widehat{\varphi}\in\operatorname{conv}\tgkla{\phi\h\nn,\phi\h\no}$.
Due to the lower bound on $E\h\trkla{\phi\h\no}$, Hölder's inequality, and \eqref{eq:regularity}, we obtain
\begin{multline}
\expected{\rkla{\sum_{n=1}^N \abs{R_{2,n}}}^\p}\leq\, C\rkla{\expected{\rkla{\tau^{1/2-\delta/2}\sum_{n=1}^N\norm{\phi\h\nn-\phi\h\no}^2_{H^1\trkla{\Om}}}^{2\p}}}^{1/2}\\
\,+C\rkla{\expected{\rkla{\tau^{-\trkla{1+1/\delta-3/2+\delta/2}}\sum_{n=1}^N\tau\norm{\phi\h\nn-\phi\h\no}_{L^2\trkla{\Om}}^{2/\delta}}^{2\p}}}^{1/2}\leq C\tau^{\p\trkla{1/2-\delta/2}}\,,
\end{multline}
which concludes the proof of \eqref{eq:errorr}.

\end{proof}

\section{Proof of the main result}\label{sec:error}
We are now in the position to prove Theorem \ref{thm:error}.
We define $\phicont\nn:=\phicont\trkla{t\nn}$ and $\mathfrak{e}\h\nn=\phicont\nn-\phi\h\nn$. 
Thus, substracting \eqref{eq:discschemesimple} from \eqref{eq:allencahn} and choosing $\proj{\err\nn}\in\Uh$ as a test function on the time interval $\trkla{t\no,t\nn}$, we obtain
\begin{align}\label{eq:difference}
\begin{split}
\iO&\rkla{\err\nn-\err\no}\proj{\err\nn}\dx + \int_{t\no}^{t\nn}\iO\nabla\err\nn\cdot\nabla\proj{\err\nn}\dx\dt\\
&=\int_{t\no}^{t\nn}\iO\nabla\rkla{\phicont\nn-\phicont}\cdot\nabla\proj{\err\nn}\dx\dt-\int_{t\no}^{t\nn}\iO\trkla{\ids-\Ihop}\gkla{\Delta\h\phi\h\nn \proj{\err\nn}}\dx\dt\\
&\quad+\int_{t\no}^{t\nn}\iO \rkla{\Ih{F^\prime\trkla{\phi\h\no}} -F^\prime\trkla{\phicont}}\proj{\err\nn}\dx\dt\\
&\quad + \int_{t\no}^{t\nn}\frac{r\h\nn-\sqrt{E\h\trkla{\phi\h\no}}}{\sqrt{E\h\trkla{\phi\h\no}}}\iO\Ih{F^\prime\trkla{\phi\h\no}}{\err\nn}\dx\dt+\int_{t\no}^{t\nn}\iO\Xi\h\nn{\err\nn}\dx\dt \\
&\quad+\sum_{k\in\mathds{Z}}\int_{t\no}^{t\nn}\iO\lambda_k\ekla{\Phi\trkla{\phicont}\g{k}-\Phi\h\trkla{\phi\h\no}\g{k}}\proj{\err\nn}\dx\dbeta_k\\
&=:\mathfrak{R}\nn_1+\mathfrak{R}\nn_2+\mathfrak{R}\nn_3+\mathfrak{R}\nn_4+\mathfrak{R}\nn_5+\mathfrak{R}\nn_6\,
\end{split}
\end{align}
$\Prob$-almost surely for all $1\leq n\leq N$.
We start by rewriting the left-hand side of \eqref{eq:difference} using the equivalence
\begin{multline}
\iO\rkla{\err\nn-\err\no}\proj{\err\nn}\dx\\
=\tfrac12\norm{\proj{\err\nn}}_{L^2\trkla{\Om}}^2+\tfrac12\norm{\proj{\err\nn-\err\no}}_{L^2\trkla{\Om}}^2-\tfrac12\norm{\proj{\err\no}}_{L^2\trkla{\Om}}^2
\end{multline}
and the estimate
\begin{align}
\begin{split}
&\int_{t\no}^{t\nn}\iO\nabla\err\nn\cdot\nabla\proj{\err\nn}\dx\dt\\
&=\int_{t\no}^{t\nn}\norm{\nabla\err\nn}_{L^2\trkla{\Om}}^2\dt -\int_{t\no}^{t\nn}\iO\nabla\err\nn\cdot\nabla\trkla{\phicont\nn-\proj{\phicont\nn}}\dx\dt\\
&\geq \trkla{1-\alpha}\int_{t\no}^{t\nn}\norm{\nabla \err\nn}_{L^2\trkla{\Om}}^2\dt -C\int_{t\no}^{t\nn} \esssup_{s\in\tekla{t\no,t\nn}} \norm{\nabla\phicont\trkla{s}-\nabla\proj{\phicont\trkla{s}}}_{L^2\trkla{\Om}}^2\dt \\
&\geq \trkla{1-\alpha}\int_{t\no}^{t\nn}\norm{\nabla \err\nn}_{L^2\trkla{\Om}}^2\dt -Ch^2\int_{t\no}^{t\nn} \esssup_{s\in\tekla{t\no,t\nn}} \norm{\Delta \phicont\trkla{s}}_{L^2\trkla{\Om}}^2\dt \,,
\end{split}
\end{align}
which is obtained using Young's inequality with $0<\alpha<\!\!<1$ and the error estimate \eqref{eq:error:proj}.
Hence, summing from $n=1$ to $M\leq N$, taking the expected value, and using the estimate
\begin{align}
\begin{split}
\tfrac14\norm{\err^M}_{L^2\trkla{\Om}}^2&\leq \tfrac12\norm{\proj{\err^M}}_{L^2\trkla{\Om}}^2+\tfrac12\norm{\phicont^M-\proj{\phicont^M}}_{L^2\trkla{\Om}}^2\\ 
&\leq \tfrac12\norm{\proj{\err^M}}_{L^2\trkla{\Om}}^2 + Ch^2\norm{\phicont^M}_{H^1\trkla{\Om}}^2
\end{split}
\end{align}
 provides
\begin{align}\label{eq:differenceSum}
\begin{split}
\tfrac14&\expected{\norm{\err^M}_{L^2\trkla{\Om}}^2}+\tfrac12\sum_{n=1}^M\expected{\norm{\proj{\err\nn-\err\no}}_{L^2\trkla{\Om}}^2} +\trkla{1-\alpha}\expected{\sum_{n=1}^M \tau \norm{\nabla\err\nn}_{L^2\trkla{\Om}}^2} \\
&\leq \tfrac12\norm{\proj{\err^0}}_{L^2\trkla{\Om}}^2 + Ch^2\expected{\norm{\phicont^M}_{H^1\trkla{\Om}}^2} +CTh^2\expected{\esssup_{s\in\tekla{0,T}}\norm{\Delta\phicont\trkla{s}}_{L^2\trkla{\Om}}^2}\\
&\qquad+\sum_{i=1}^6\expected{\sum_{n=1}^M\mathfrak{R}_i\nn}\,.
\end{split}
\end{align}
We shall now derive estimates for the remaining terms.
For the first remaining error term, we compute using Young's inequality with $0<\alpha<\!\!<1$, the stability of $\projop$ with respect to the $H^1$-seminorm, and \eqref{eq:regu:hoelder}
\begin{align}
\begin{split}
\sum_{n=1}^M\expected{\mathfrak{R}_1^n}\leq&\, \alpha\expected{\sum_{n=1}^M\tau\norm{\nabla\err\nn}_{L^2\trkla{\Om}}^2}+C\sum_{n=1}^N\expected{\int_{t\no}^{t\nn}\norm{\nabla u\nn-\nabla u}_{L^2\trkla{\Om}}^2\dt}\\
\leq&\, \alpha\expected{\sum_{n=1}^M\tau\norm{\nabla\err\nn}_{L^2\trkla{\Om}}^2}+C\tau\,.
\end{split}
\end{align}
Applying Lemma \ref{lem:Ihestimates}, an inverse inequality, Young's inequality, the stability of $\projop$, and Lemma \ref{lem:regdisc}, we derive for the second term
\begin{align}
\begin{split}
\sum_{n=1}^M\expected{\mathfrak{R}_2^n}\leq&\, Ch^2 \expected{\sum_{n=1}^N\tau \norm{\Delta\h \phi\h\nn}_{L^2\trkla{\Om}}^2} + \alpha\sum_{n=1}^M\tau\expected{\norm{\nabla \err\nn}_{L^2\trkla{\Om}}^2}\\
\leq&\, Ch^2 + \alpha\sum_{n=1}^M\tau\expected{\norm{\nabla \err\nn}_{L^2\trkla{\Om}}^2}\,.
\end{split}
\end{align}
To obtain an estimate for $\mathfrak{R}_{3}\nn$, we recall $F^\prime\trkla{s}=s^3-s$ and write
\begin{multline}
\rkla{\Ih{F^\prime\trkla{\phi\h\no}}-F^\prime\trkla{\phicont}}\proj{\err\nn}=\err\nn\proj{\err\nn} + \rkla{-\trkla{\phicont\nn}^3+\trkla{\phi\h\nn}^3}\rkla{\phicont\nn-\phi\h\nn}\\
+\rkla{\trkla{\phicont\nn}^3-\trkla{\phi\h\nn}^3}\rkla{\phicont\nn-\proj{\phicont\nn}} +\rkla{F^\prime\trkla{\phicont\nn}-F^\prime\trkla{\phicont}}\proj{\err\nn} \\
+\rkla{F^\prime\trkla{\phi\h\no}-F^\prime\trkla{\phi\h\nn}}\proj{\err\nn} -\trkla{\ids-\Ihop}\gkla{F^\prime\trkla{\phi\h\no}}\proj{\err\nn}\\
=:\mathfrak{I}_1+\mathfrak{I}_2+\mathfrak{I}_3+\mathfrak{I}_4+\mathfrak{I}_5+\mathfrak{I}_6\,.
\end{multline}
Straightforward computations show $\iO\mathfrak{I}_1\dx\leq \norm{\err\nn}_{L^2\trkla{\Om}}^2$ and $\iO \mathfrak{I}_2\dx\leq 0$.
For the third term, we compute using Sobolev embeddings, Young's inequality, and \eqref{eq:error:proj}
\begin{align}
\begin{split}
\iO\mathfrak{I}_3\dx=&\,\iO \rkla{\tabs{\phicont\nn}^2+\phicont\nn\phi\h\nn+\tabs{\phi\h\nn}^2}\rkla{\phicont\nn-\phi\h\nn}\rkla{\phicont\nn-\proj{\phicont\nn}}\dx\\
\leq&\,C\rkla{\norm{\phicont\nn}_{H^1\trkla{\Om}}^2+\norm{\phi\h\nn}_{H^1\trkla{\Om}}^2}\norm{\err\nn}_{H^1\trkla{\Om}}\norm{\phicont\nn-\proj{\phicont\nn}}_{L^2\trkla{\Om}}\\
\leq&\,\alpha\norm{\err\nn}_{H^1\trkla{\Om}}^2+C\rkla{\norm{\phicont\nn}_{H^1\trkla{\Om}}^4+\norm{\phi\h\nn}_{H^1\trkla{\Om}}^4}\norm{\phicont\nn-\proj{\phicont\nn}}_{L^2\trkla{\Om}}^2\\
\leq &\, \alpha\norm{\err\nn}_{H^1\trkla{\Om}}^2 +Ch^2 \rkla{\norm{\phicont\nn}_{H^1\trkla{\Om}}^4+\norm{\phi\h\nn}_{H^1\trkla{\Om}}^4}\norm{ \phicont\nn}_{H^1\trkla{\Om}}^2\,.
\end{split}
\end{align}
Similarly, we deduce for $\mathfrak{I}_4$ using the stability of $\projop$ with respect to the $H^1$-norm
\begin{align}
\begin{split}
\iO\mathfrak{I}_4\dx\leq&\, C\rkla{\norm{\phicont\nn}_{H^1\trkla{\Om}}^2+\norm{\phicont}_{H^1\trkla{\Om}}^2+1}\norm{\phicont\nn-\phicont}_{L^2\trkla{\Om}}\norm{\proj{\err\nn}}_{H^1\trkla{\Om}}\\
\leq&\,\alpha\norm{\err\nn}_{H^1\trkla{\Om}}^2+ C\trkla{\norm{\phicont\nn}_{H^1\trkla{\Om}}^4+\norm{\phicont}_{H^1\trkla{\Om}}^4+1}\norm{\phicont\nn-\phicont}_{L^2\trkla{\Om}}^2\,
\end{split}
\end{align}
and for the fifth error term
\begin{align}
\begin{split}
\iO\mathfrak{I}_5\dx\leq&\, \alpha\norm{\err\nn}_{H^1\trkla{\Om}}^2 \\
&\,+ C\rkla{\norm{\phi\h\nn}_{H^1\trkla{\Om}}^4+\norm{\phi\h\no}_{H^1\trkla{\Om}}^4+1}\norm{\phi\h\nn-\phi\h\no}_{L^2\trkla{\Om}}^2\,.
\end{split}
\end{align}
Finally, to deal with $\mathcal{I}_6$, we use $F^\prime\trkla{s}=s^3-s$ and Young's inequality to obtain
\begin{align}
\begin{split}
\iO\mathfrak{I}_6\dx\leq C\norm{\trkla{\ids-\Ihop}\gkla{\trkla{\phi\h\no}^3}}_{L^2\trkla{\Om}}^2+\alpha\norm{\err\nn}_{L^2\trkla{\Om}}^2\,.
\end{split}
\end{align}
To estimate the first term, we compute on each $K\in\Th$ by using the standard error estimate for $\Ihop$ (see e.g.~\cite[Theorem 4.4.4]{BrennerScott})
\begin{align}
\begin{split}
&\norm{\trkla{\ids-\Ihop}\gkla{\rkla{\phi\h\no}^3}}_{L^2\trkla{K}}^2\leq \abs{K}\norm{\trkla{\ids-\Ihop}\gkla{\rkla{\phi\h\no}^3}}_{L^\infty\trkla{K}}^2\\
&\qquad\leq C\abs{K}h^2\norm{\trkla{\phi\h\no}^2\nabla\phi\h\no}_{L^\infty\trkla{K}}^2\leq Ch^2\norm{\nabla\phi\h\no}_{L^2\trkla{K}}^2\norm{\phi\h\no}_{L^\infty\trkla{\Om}}^4\\
&\qquad\leq C h^2\norm{\nabla\phi\h\no}_{L^2\trkla{K}}^2 \rkla{\norm{\Delta\h\phi\h\no}_{L^2\trkla{\Om}}^{1/2}\norm{\phi\h\no}_{H^1\trkla{\Om}}^{1/2}+\norm{\phi\h\no}_{H^1\trkla{\Om}}}^4\,.
\end{split}
\end{align}
Here, we used a discrete version of the Gagliardo--Nirenberg inequality (see e.g.~\cite[Corollary 6.1]{GrunGuillenMetzger2016} or \cite[Lemma A.1]{Metzger2020}) to control the $L^\infty\trkla{\Om}$-norm using the $H^1\trkla{\Om}$-norm and a discrete counterpart of the $H^2\trkla{\Om}$-norm.
Hence, we obtain
\begin{multline}
\int_{t\no}^{t\nn}\iO\mathfrak{I}_6\dx\dt\leq \alpha\int_{t\no}^{t\nn}\norm{\err\nn}_{L^2\trkla{\Om}}^2\dt \\
+Ch^2\rkla{\max_{0\leq m\leq N}\norm{\phi\h^m}_{H^1\trkla{\Om}}^4}\int_{t\no}^{t\nn}\norm{\Delta\h\phi\h\no}_{L^2\trkla{\Om}}^2+\norm{\phi\h\no}_{H^1\trkla{\Om}}^2\dt\,.
\end{multline}
Summarizing the above estimates, we deduce for any $0<\delta<1$
\begin{align}
\begin{split}
\expected{\sum_{n=1}^M\mathfrak{R}_3\nn}\leq&\, \trkla{1+4\alpha}\sum_{n=1}^M\tau\expected{\norm{\err\nn}_{L^2\trkla{\Om}}^2} +3\alpha \expected{\sum_{n=1}^M\tau\norm{\nabla\err\nn}_{L^2\trkla{\Om}}^2}\\
&\,+Ch^2\expected{\sum_{n=1}^N\tau\rkla{\norm{\phicont\nn}_{H^1\trkla{\Om}}^6+\norm{\phi\h\nn}_{H^1\trkla{\Om}}^6}}\\
&\,+C\expected{\sum_{n=1}^N\int_{t\no}^{t\nn}\rkla{\norm{\phicont\nn}_{H^1\trkla{\Om}}^4+\norm{\phicont}_{H^1\trkla{\Om}}^4+1}\norm{\phicont\nn-\phicont}_{L^2\trkla{\Om}}^2\dt}\\
&\,+C\expected{\sum_{n=1}^N\tau\rkla{\norm{\phi\h\nn}_{H^1\trkla{\Om}}^4+\norm{\phi\h\no}_{H^1\trkla{\Om}}^4+1}\norm{\phi\h\nn-\phi\h\no}_{L^2\trkla{\Om}}^2}\\
&\,+Ch^2\expected{\max_{0\leq m\leq N}\norm{\phi\h^m}_{H^1\trkla{\Om}}^4 \sum_{n=1}^N\tau\rkla{\norm{\Delta\h\phi\h\no}_{L^2\trkla{\Om}}^2+\norm{\phi\h\no}_{H^1\trkla{\Om}}^2}}\\
\leq&\,\trkla{1+4\alpha}\sum_{n=1}^M\tau\expected{\norm{\err\nn}_{L^2\trkla{\Om}}^2} +3\alpha \expected{\sum_{n=1}^M\tau\norm{\nabla\err\nn}_{L^2\trkla{\Om}}^2}\\
&\,+Ch^2 +C\tau^{1-\delta} +C\tau +Ch^2\,,
\end{split}
\end{align}
due to Young's inequality, Hölder's inequality, and the regularity results stated in Lemma \ref{lem:regu} and Lemma \ref{lem:regdisc}.
To deal with the fourth remaining error term on the right-hand side of \eqref{eq:differenceSum}, we apply Young's inequality, Hölder's inequality, and the results collected in Lemma \ref{lem:regdisc} to deduce for $0<\delta<1$
\begin{align}
\begin{split}
\expected{\sum_{n=1}^M\mathfrak{R}_4\nn}&\leq\alpha\expected{\sum_{n=1}^M\tau\norm{\err\nn}_{L^2\trkla{\Om}}^2} \\
&~+C\expected{\max_{0\leq m\leq N}\norm{\phi\h^m}_{H^1\trkla{\Om}}^6\sum_{n=1}^N\tau\rkla{\abs{r\h\nn-r\h\no}^2+\abs{r\h\no-\sqrt{E\trkla{\phi\h\no}}}^2}}\\
&\leq\alpha\sum_{n=1}^M\tau\expected{\norm{\err\nn}_{L^2\trkla{\Om}}^2} +C\tau +C\tau^{1-\delta}\,.
\end{split}
\end{align}
The fifth term can be estimated using Young's inequality and Lemma \ref{lem:regdisc}:
\begin{align}
\expected{\sum_{n=1}^M\mathfrak{R}_5\nn}\leq \alpha\sum_{n=1}^M\tau\expected{\norm{\err\nn}_{L^2\trkla{\Om}}^2}+C\tau\,.
\end{align}
%
%
%
For the expected value of $\mathfrak{R}_6$, we deduce by adding and subtracting $\proj{\err\no}$ and applying Young's inequality and \Ito's isometry
\begin{align}
\begin{split}
\expected{\mathfrak{R}_6\nn} =&\,\expected{\sum_{k\in\mathds{Z}} \int_{t\no}^{t\nn}\iO \lambda_k\ekla{\Phi\trkla{\phicont}\g{k}-\Phi\h\trkla{\phi\h\no}\g{k}}\proj{\err\no}\dx\dbeta_k}\\
&\,+\expected{\sum_{k\in\mathds{Z}} \int_{t\no}^{t\nn}\iO \lambda_k\ekla{\Phi\trkla{\phicont}\g{k}-\Phi\h\trkla{\phi\h\no}\g{k}}\proj{\err\nn-\err\no}\dx\dbeta_k} \\
\leq&\,\tfrac14\expected{\norm{\proj{\err\nn-\err\no}}_{L^2\trkla{\Om}}^2}\\
&\,+\expected{\sum_{k\in\mathds{Z}}\int_{t\no}^{t\nn}\iO\lambda_k^2\ekla{\Phi\trkla{\phicont}\g{k}-\Phi\h\trkla{\phi\h\no}\g{k}}^2\dx\dt}\,.
\end{split}
\end{align}
We continue by estimating the second term on the right-hand side.
As $\Phi\h\trkla{\phi\h\no}\g{k}=0$ for all $k\not\in\mathds{Z}\h$, we can use assumptions \ref{item:C} and \ref{item:Z} to show
\begin{align}
\expected{\sum_{k\in\mathds{Z}\setminus\mathds{Z}\h}\int_{t\no}^{t\nn}\iO\rkla{\lambda_k\varrho\trkla{\phicont}\g{k}}^2\dx\dt}\leq C\tau\sum_{k\in\mathds{Z}\setminus\mathds{Z}\h}\norm{\lambda_k\g{k}}_{L^2\trkla{\Om}}^2\leq C\tau h^2\,.
\end{align}
To deal with the remaining part, we estimate
\begin{multline}
\expected{\sum_{k\in\mathds{Z}\h}\int_{t\no}^{t\nn}\iO\lambda_k^2\rkla{\varrho\trkla{\phicont}\g{k}-\Ih{\varrho\trkla{\phi\h\no}\g{k}}}^2\dx\dt}\\
\leq C\expected{\sum_{k\in\mathds{Z}\h}\int_{t\no}^{t\nn}\iO\lambda_k^2\abs{\rkla{\varrho\trkla{\phicont}-\varrho\trkla{\phi\h\no}}\g{k}}^2\dx\dt}\\
+C\expected{\sum_{k\in\mathds{Z}\h}\int_{t\no}^{t\nn}\iO\lambda_k^2\abs{\trkla{\ids-\Ihop}\gkla{\varrho\trkla{\phi\h\no}\g{k}}}^2\dx\dt}=:\mathfrak{J}_1+\mathfrak{J}_2
\end{multline}
To obtain an estimate for $\mathfrak{J}_2$, we again use the error estimates for $\Ihop$ on each $K\in\Th$ individually and obtain due to assumption \ref{item:C}
\begin{align}
\begin{split}
\lambda_k^2&\,\int_K\abs{\trkla{\ids-\Ihop}\gkla{\varrho\trkla{\phi\h\no}\g{k}}}^2\dx\leq \lambda_k^2\abs{K}\norm{\trkla{\ids-\Ihop}\gkla{\varrho\trkla{\phi\h\no}\g{k}}}_{L^\infty\trkla{K}}^2\\
&\leq C\lambda_k^2h^2\abs{K}\rkla{\norm{\nabla\phi\h\no}_{L^\infty\trkla{K}}^2\norm{\g{k}}_{L^\infty\trkla{K}}^2+\norm{\nabla\g{k}}_{L^\infty\trkla{K}}^2}\,.
\end{split}
\end{align}
Hence, using $\abs{K}\norm{\nabla\phi\h\no}_{L^\infty\trkla{K}}^2=\norm{\nabla\phi\h\no}_{L^2\trkla{K}}^2$ and assumption \ref{item:W2}, we obtain
\begin{align}
\mathfrak{J}_2\leq C h^2 \expected{\int_{t\no}^{t\nn}\norm{\nabla\phi\h\no}_{L^2\trkla{\Om}}^2+1\dt}\,.
\end{align}
Due to the Lipschitz continuity of $\varrho$ (cf.~\ref{item:C}), we deduce for $\mathfrak{J}_1$ using \ref{item:W2}
\begin{align}
\begin{split}
\mathfrak{J}_1\leq&\,  C\expected{\int_{t\no}^{t\nn}\norm{\phicont-\phicont\no}_{L^2\trkla{\Om}}^2 +\norm{\err\no}_{L^2\trkla{\Om}}^2\dt}\,.
\end{split}
\end{align}
Hence, we obtain for any $0<\delta<1$ using Lemma \ref{lem:regu} 
\begin{align}
\begin{split}
\sum_{n=1}^M\expected{\mathfrak{R}_6\nn}\leq&\,\tfrac14\sum_{n=1}^M\expected{\norm{\proj{\err\nn-\err\no}}_{L^2\trkla{\Om}}^2}+C\sum_{n=1}^M\tau\expected{\norm{\err\no}_{L^2\trkla{\Om}}^2} \\
&\,+C\sum_{n=1}^M\expected{\int_{t\no}^{t\nn}\norm{\phicont-\phicont\no}_{L^2\trkla{\Om}}^2\dt} +Ch^2\\
\leq&\,\tfrac14\sum_{n=1}^M\expected{\norm{\proj{\err\nn-\err\no}}_{L^2\trkla{\Om}}^2}+C\sum_{n=1}^M\tau\expected{\norm{\err\no}_{L^2\trkla{\Om}}^2}  \\
&\,+C\tau^{1-\delta}+Ch^2\,.
\end{split}
\end{align}
Combining the above results, an absorption argument provides
\begin{multline}
\rkla{\tfrac14-\trkla{1+6\alpha}\tau}\expected{\norm{\err^M}_{L^2\trkla{\Om}}^2}+\trkla{1-6\alpha}\tau\sum_{n=1}^M\expected{\norm{\err\nn}_{L^2\trkla{\Om}}^2}\\
\leq \tfrac12\norm{\proj{\err^0}}_{L^2\trkla{\Om}}^2 +C\sum_{n=1}^M\tau\expected{\norm{\err\no}_{L^2\trkla{\Om}}^2} +Ch^2+C\tau^{1-\delta}\,.
\end{multline}
Hence, recalling \ref{item:I} and choosing $\alpha>0$ and $\tau>0$ sufficiently small, the application of a discrete version of Gronwall's inequality (see e.g.~\cite[Lemma 10.5]{Thomee2006}) completes the proof of Theorem \ref{thm:error}.

\section{Numerical simulations}\label{sec:numerics}
In this section, we demonstrate the practicality of the augmented SAV scheme \eqref{eq:discscheme} and compute an experimental order of convergence with respect to the discretization parameters $h$ and $\tau$.
Using our \texttt{C++} framework EconDrop (cf.~\cite{Campillo2012}, \cite{Grun2013c}, \cite{Metzger2023}, \cite{Metzger2024}, and the references therein), we simulate the evolution of an elliptical shaped droplet in the periodic domain $\Om:=\trkla{0,1}^2$ and the time interval $\tekla{0,T}$ with $T=1.04$. 
The parameters are chosen as follows:
The width of the diffuse interface is governed by $\varepsilon=0.02$ and the double-well potential is shifted by $10^{-5}$ to guarantee positivity, i.e.~$F\trkla{\phi}:=\tfrac14\trkla{\phi^2-1}^2+10^{-5}$.
For our simulations, we consider a finite dimensional Wiener process of the form
\begin{align}
W\trkla{x,y,t}:=\sum_{k=-3}^{3}\sum_{l=-3}^3\lambda_k\lambda_l\g{k}\trkla{x}\g{l}\trkla{y}\beta_{kl}\trkla{t}\,
\end{align}
with $\lambda_0=\lambda_{\pm1}=1$, $\lambda_{\pm2}=1/4$, and $\lambda_{\pm3}=1/9$.
Here, $\trkla{\g{k}}_{k\in\mathds{Z}}$ are chosen as eigenfunctions of the one-dimensional Laplacian with periodic boundary conditions, i.e.
\begin{align}
\g{k}\trkla{x}:=\left\{\begin{array}{cl}
\sqrt{2}\cos\trkla{2\pi kx}&\text{for~}k\geq1\,,\\
1&\text{for~}k=0\,,\\
\sqrt{2}\sin\trkla{2\pi kx}&\text{for~}k\leq-1\,.
\end{array}\right.
\end{align}
Hence, the products $\trkla{\g{k}\trkla{x}\g{l}\trkla{y}}_{kl}$ form an an orthonormal basis of $L^2\trkla{\Om}$.
The mutually independent Brownian motions $\trkla{\beta_{kl}}_{kl}$ are  evaluated using the smallest considered time increment $\tau_{\mathrm{min}}=10^{-5}$ and the \texttt{xoshiro256plus} generator from version 4.24 of the TRNG library (cf.~\cite{BaukeMertens2007}).
$\Phi\h$ is defined as in \eqref{eq:def:Phih} with an indicator function $\varrho\trkla{\phi}:=\trkla{2\sqrt{\varepsilon}}^{-1}\max\tgkla{1-\phi^2,0}$, i.e.~the stochastic source or sink terms affect in the interfacial region.\\
In our numerical experiment we consider an elliptical shaped droplet centered at $(0.5,0.5)$ with diameter $0.6$ in $x$-direction and diameter $0.36$ in $y$-direction.
Figure \ref{fig:evolution} illustrates the evolution of this droplet computed using the smallest considered discretization parameters $\tau_{\mathrm{min}}=10^{-5}$ and $h_{\mathrm{min}}=2^{-8}$.
Here, the left-column depicts the expected value of the phase-field which is computed using 350 individual sample paths.
The remaining three columns show the evolution of different individual paths.
To allow for a comparison with the deterministic case, the zero level set of the deterministic solution is depicted in orange.
As shown in Figure \ref{fig:evolution}, the stochastic source terms do not only lead to slight oscillations of the deterministic solution, but are chosen sufficiently strong to counteract the regularizing effects of the mean curvature flow and results in significant variations in shape, size, and position of the droplets.\\
To validate the theoretical convergence results stated in Theorem \ref{thm:error}, we compute an experimental order of convergence.
For this reason, we compare the solutions obtained for different discretization parameters on a grid using $\widetilde{h}=2^{-8}$ and $\widetilde{\tau}=3.2\cdot10^{-3}$ with our reference solution $\phi_*$, which is computed using $h_{\min}=2^{-8}$ and $\tau_{\min}=10^{-5}$.
In particular, we shall evaluate the following discrete error norms
\begin{subequations}
\begin{align}
E_{h,\tau}^{L^2}:=&\rkla{\max_{1\leq n\leq\widetilde{N}}\widetilde{\mathds{E}}_{350}\ekla{\norm{\phi\h\nn -\phi_*\trkla{t\nn}}_{L^2\trkla{\Om}}^2}}^{1/2}\,,\\
E_{h,\tau}^{H^1}:=&\rkla{\widetilde{\tau}\sum_{n=1}^{\widetilde{N}}\widetilde{\mathds{E}}_{350}\ekla{\norm{\phi\h\nn-\phi_*\trkla{t\nn}}_{H^1\trkla{\Om}}^2}}^{1/2}\,
\end{align}
\end{subequations}
with $\widetilde{N}=T/\widetilde{\tau}$.
Here, $\widetilde{E}_{350}$ denoting the average of 350 samples serves as an approximation of the expected value.
Recalling the convergence result stated in Theorem \ref{thm:error}, we use the scaling $\tau\sim h^2$.
The results listed in Table \ref{tab:error} show that even for the coarsest considered discretization parameters the experimental order of convergence is only slightly below $1/2$.
Further, the results indicate that the error and the order of convergence is governed by the deviation with respect to the $L^2\trkla{\Omega;L^2\trkla{0,T;H^1\trkla{\Om}}}$-norm.

\begin{figure}
\begin{center}
\newcommand{\scale}{0.23}
\subfloat[][$t=0.08$]{
\includegraphics[width=\scale\textwidth]{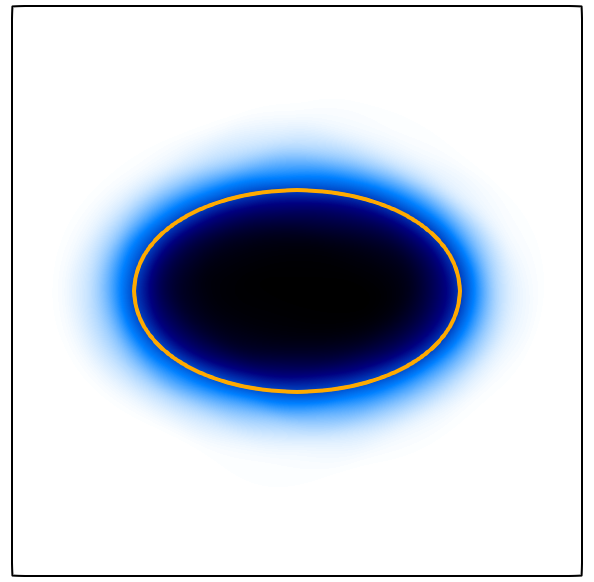}\hfill
\includegraphics[width=\scale\textwidth]{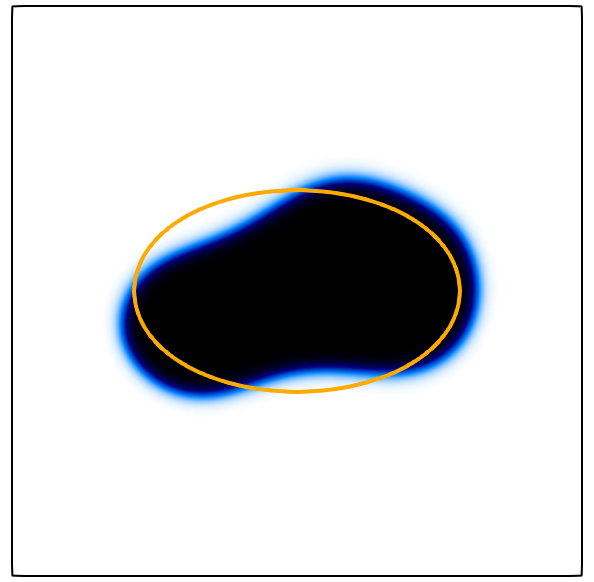}\hfill
\includegraphics[width=\scale\textwidth]{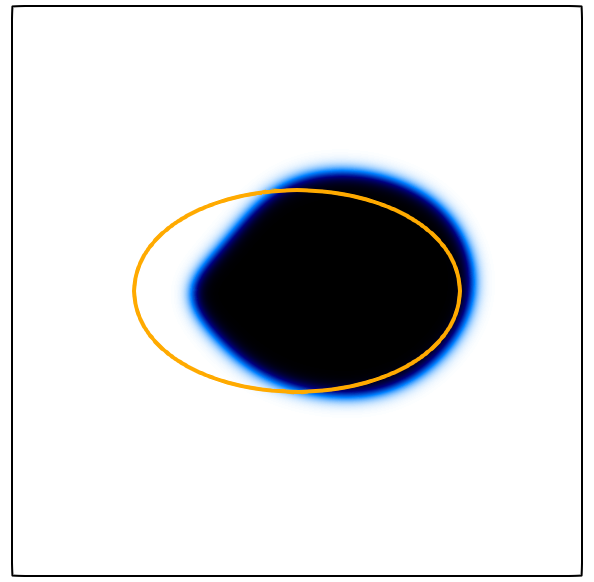}\hfill
\includegraphics[width=\scale\textwidth]{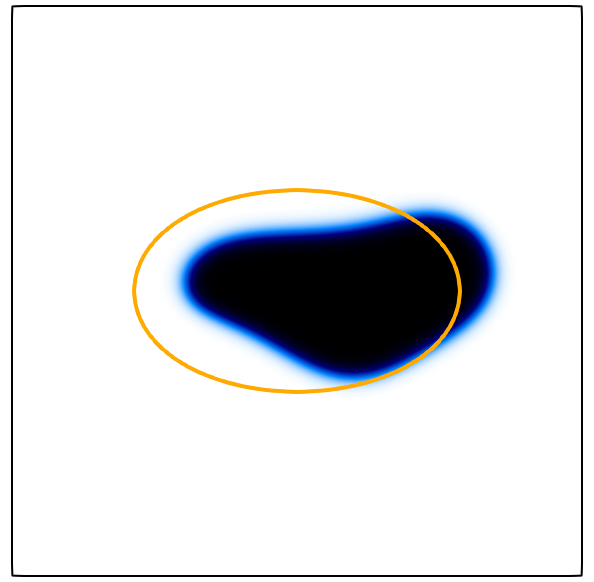}
}\\

\subfloat[][$t=0.4$]{
\includegraphics[width=\scale\textwidth]{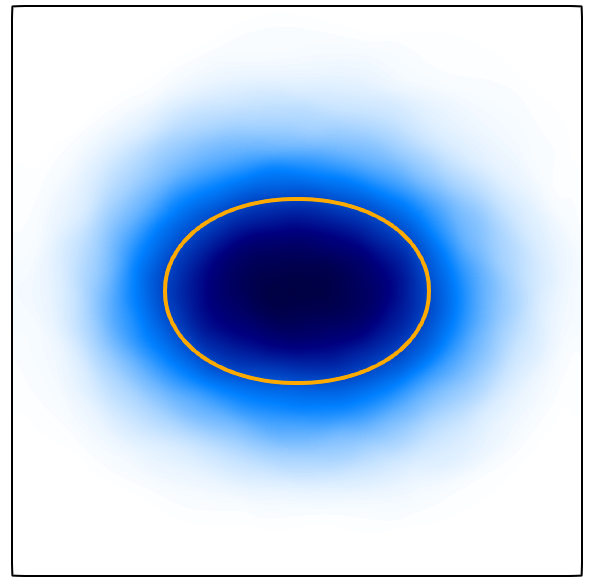}\hfill
\includegraphics[width=\scale\textwidth]{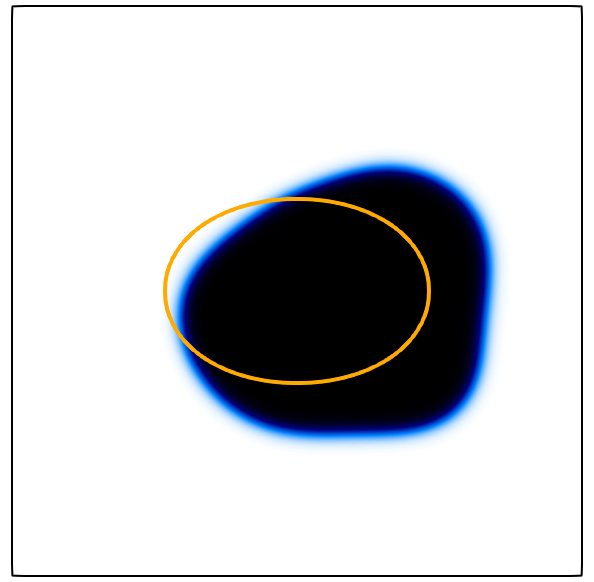}\hfill
\includegraphics[width=\scale\textwidth]{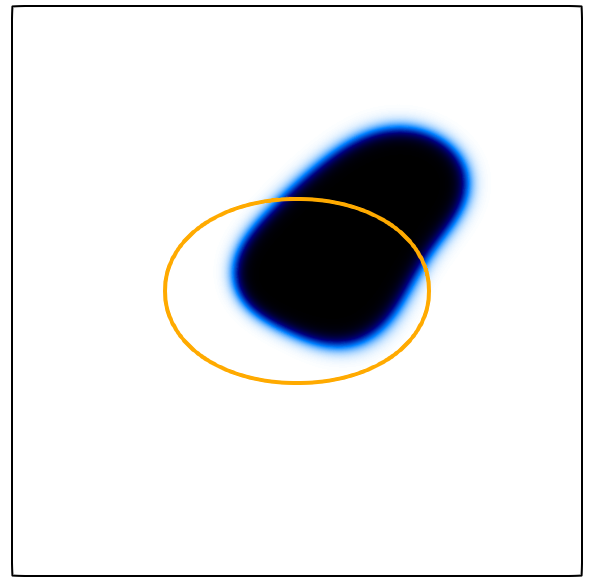}\hfill
\includegraphics[width=\scale\textwidth]{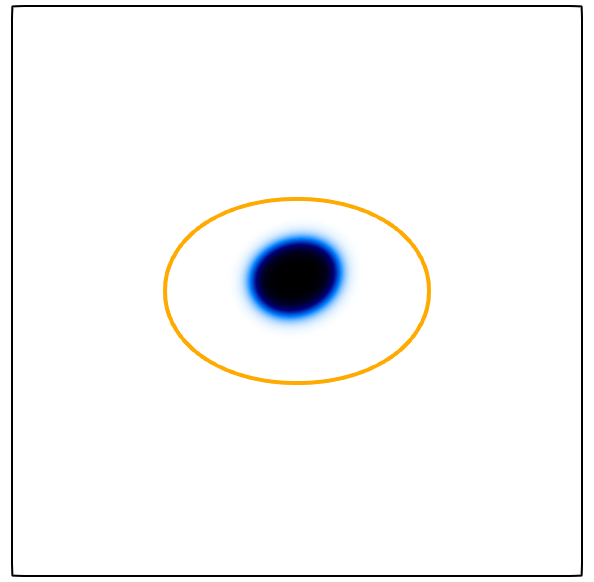}
}\\

\subfloat[][$t=0.8$]{
\includegraphics[width=\scale\textwidth]{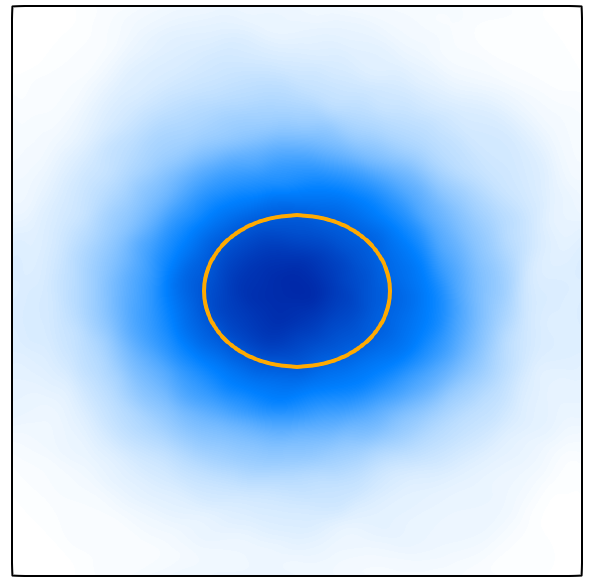}\hfill
\includegraphics[width=\scale\textwidth]{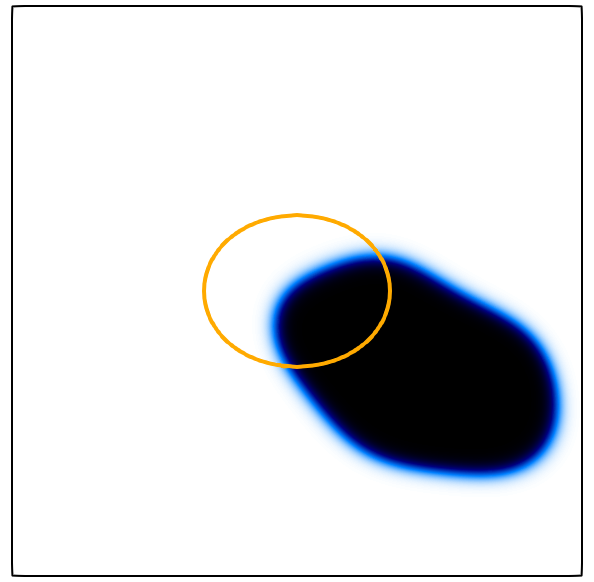}\hfill
\includegraphics[width=\scale\textwidth]{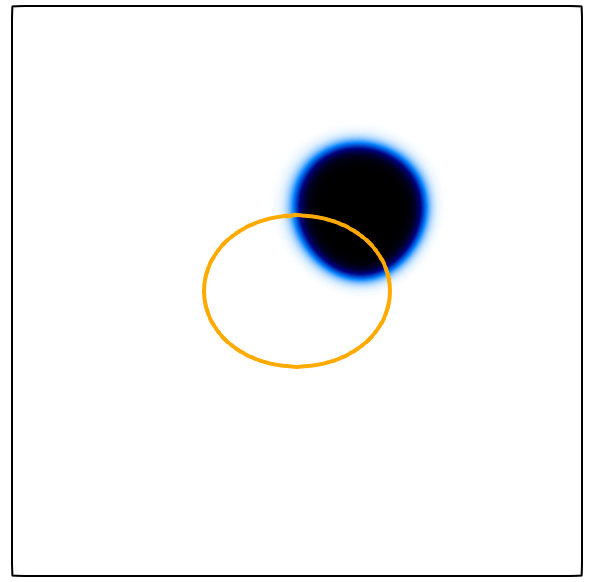}\hfill
\includegraphics[width=\scale\textwidth]{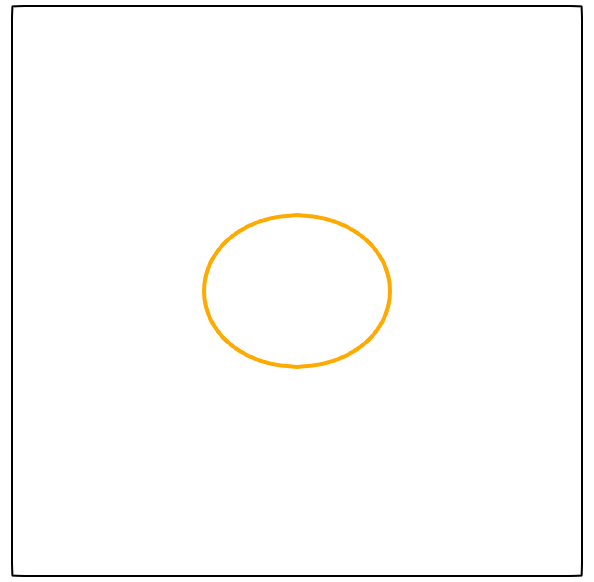}
}\\

\subfloat[][$t=1.04$]{
\includegraphics[width=\scale\textwidth]{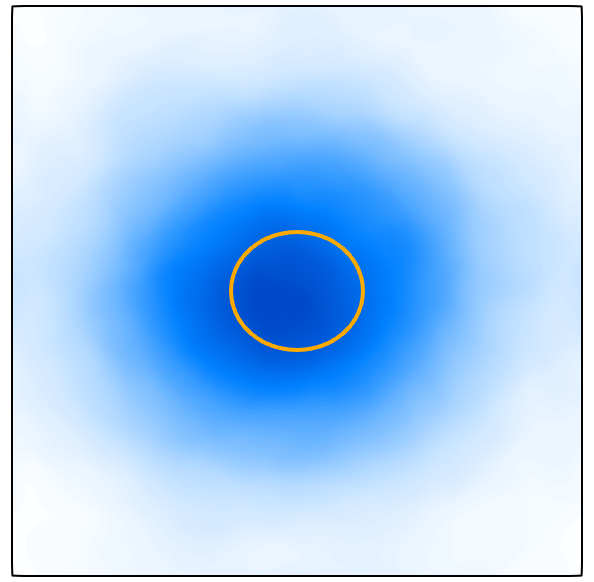}\hfill
\includegraphics[width=\scale\textwidth]{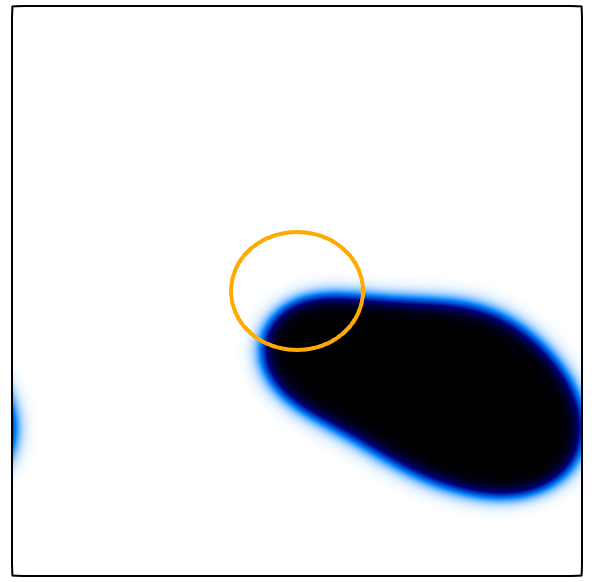}\hfill
\includegraphics[width=\scale\textwidth]{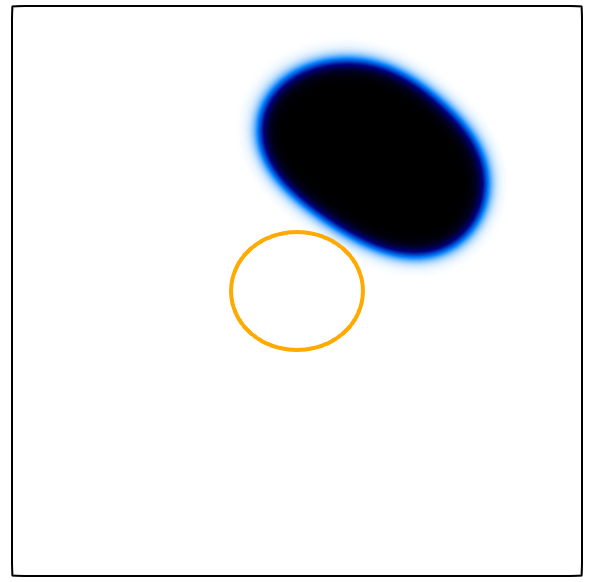}\hfill
\includegraphics[width=\scale\textwidth]{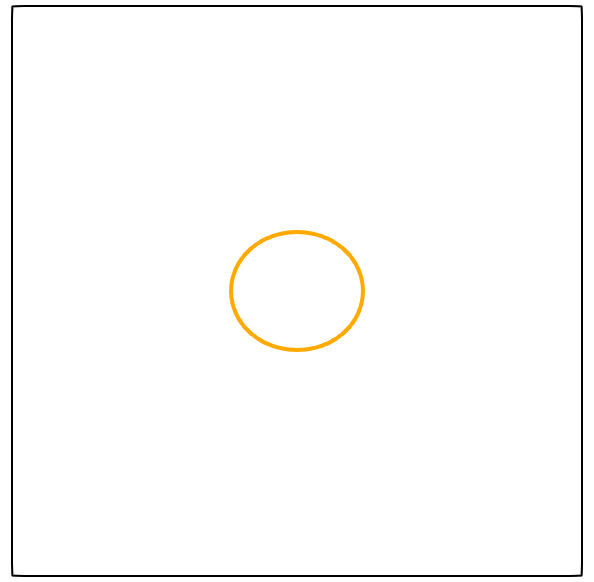}
}
\end{center}
\caption{Evolution of a droplet computed with $\tau=1\cdot10^{-5}$ and $h=2^{-8}$. The left column depicts the expected value, while the remaining columns illustrate three individual sample paths.
The orange line indicates the zero level set of the deterministic evolution.}
\label{fig:evolution}
\end{figure}

\begin{table}
\begin{tabular}{lr|ll|ll|cl}
~~$h$&$\tau$\phantom{100}& $E_{h,\tau}^{L^2}$ & EOC & $E_{h,\tau}^{H^1}$ & EOC&$\ekla{\rkla{E_{h,\tau}^{L^2}}^2+\rkla{E_{h,\tau}^{H^1}}^2}^{1/2}$&EOC\\
$2^{-7.5}$&$2\cdot10^{-5} $ &0.08& --- & 1.79 & ---& 1.79 &---     \\
$2^{-7}$&$4\cdot10^{-5}$    &0.17& 1.07  &3.50 & 0.97 & 3.50 & 0.97\\
$2^{-6.5}$&$8\cdot10^{-5}$  &0.28& 0.72 &5.54 & 0.66 & 5.55 &0.66  \\
$2^{-6}$&$16\cdot10^{-5}$   &0.36& 0.39 &7.61 & 0.46 & 7.62 &0.46   
\end{tabular}
\caption{Experimental order of convergence w.r.t.~$\trkla{\tau,h^2}$.}
\label{tab:error}
\end{table}

\section{Conclusion}\label{sec:conclusion}
We analyzed an augmented SAV discretization for the stochastic Allen--Cahn equation with multiplicative noise.
This scheme has the advantage of being linear with respect to the unknown quantities while still resembling the energetic structure of the original differential equation sufficiently well to allow for unconditional stability results making it a powerful tool for the numerical treatment of stochastic PDEs.
Due to the specific choice of the augmentation terms we are able to show that the difference between the scalar auxiliary variable and $\sqrt{\iO F\trkla{\phi}\dx}$ is of order $\tau^{1/2-\delta/2}$ for any $\delta>0$.
Hence, we were able to show that the convergence rate of our scheme is on a par with the optimal strong rates $\tau^{1/2-\delta/2}+h^2$ of convergence established in \cite{MajeeProhl2018} for an implicit, nonlinear discrete scheme.


\bibliographystyle{amsplain}
\providecommand{\bysame}{\leavevmode\hbox to3em{\hrulefill}\thinspace}
\providecommand{\MR}{\relax\ifhmode\unskip\space\fi MR }
\providecommand{\MRhref}[2]{%
  \href{http://www.ams.org/mathscinet-getitem?mr=#1}{#2}
}
\providecommand{\href}[2]{#2}

\end{document}